\documentclass[12pt]{amsart}

\usepackage{amsmath}
\usepackage[mathscr]{euscript}
\usepackage{amssymb}

\usepackage{quiver}
\usepackage{tikz-cd}

\usepackage[margin=1in]{geometry}  
\usepackage{hyperref} 
\hypersetup{colorlinks=true,linkcolor=blue,citecolor=magenta}

\theoremstyle{plain}
\newtheorem{theorem}[equation]{Theorem}
\newtheorem{lemma}[equation]{Lemma}
\newtheorem{corollary}[equation]{Corollary}
\newtheorem{proposition}[equation]{Proposition}

\theoremstyle{definition}
\newtheorem{definition}[equation]{Definition}

\newtheorem{remark}[equation]{Remark}

\newtheorem{example}[equation]{Example}

\newtheorem*{intro-prop}{Proposition~\ref{prop:lim}}
\newtheorem*{intro-thm}{Theorem~\ref{thm:main}}

\newcommand{\isom}{\cong}

\newcommand{\Fun}{\operatorname{Fun} }

\newcommand{\un}[1]{\underline{#1}}
\newcommand{\bcat}[1]{\mathbb{#1}}

\newcommand{\DBun}{\mathrm{DBun}}
\newcommand{\pb}[1]{\langle #1 \rangle}
\newcommand{\CRing}{\mathbb{CR}\mathrm{ing}}
\newcommand{\Mod}{\mathrm{Mod}}
\newcommand{\Sch}{\mathbb{S}\mathrm{ch}}
\newcommand{\Alg}{\mathbb{A}\mathrm{lg}}
\newcommand{\Sym}{\mathrm{Sym}}
\newcommand{\Mfld}{\mathbb{M}\mathrm{fld}}
\newcommand{\Z}{\mathbb{Z}}

\begin{document}

\title{A characterization of differential bundles in tangent categories}

\author{Michael Ching}
\email{mching@amherst.edu}
\address{Department of Mathematics and Statistics, Amherst College, Amherst, MA 01002, USA}

\begin{abstract}
A tangent category is a categorical abstraction of the tangent bundle construction for smooth manifolds. In that context, Cockett and Cruttwell develop the notion of differential bundle which, by work of MacAdam, generalizes the notion of smooth vector bundle to the abstract setting. Here we provide a new characterization of differential bundles and show that, up to isomorphism, a differential bundle is determined by its projection map and zero section. We show how these results can be used to quickly identify differential bundles in various tangent categories.
\end{abstract}

\maketitle

\section*{Introduction}

A tangent structure on a category $\bcat{X}$ consists of a functor $T: \bcat{X} \to \bcat{X}$, which plays the role of a `tangent bundle' construction for objects in $\bcat{X}$, along with various natural transformations involving $T$. The notion is originally due to Rosick\'{y}~\cite{rosicky:1984}, but it was modified and the theory greatly extended by Cockett and Cruttwell~\cite{cockett/cruttwell:2014}. We refer the reader there for a precise definition.

Here is the idea. A tangent structure on a category $\bcat{X}$ includes, for each object $M$ of $\bcat{X}$, a morphism
\[ p_M: T(M) \to M \]
which plays the role of the tangent bundle of $M$, and another morphism
\[ 0_M: M \to T(M) \]
which plays the role of the zero section of that bundle. These morphisms are the components of natural transformations $p: T \to 1_{\bcat{X}}$ and $0: 1_{\bcat{X}} \to T$, respectively. The canonical example is $\bcat{X} = \Mfld$, the category of smooth manifolds and smooth maps, in which case $p_M$ and $0_M$ are precisely the projection and zero section of the ordinary smooth tangent bundle. Other parts of the structure of a tangent category describe additional aspects of the basic theory of manifolds such as the addition of tangent vectors in the same fibre.

Cockett and Cruttwell develop in~\cite{cockett/cruttwell:2018} a notion of `differential bundle' in a tangent category, which is intended to capture in this categorical framework the smooth vector bundles in the category of manifolds. A differential bundle consists of four morphisms:
\begin{itemize} \itemsep=5pt
\item $q: E \to M$, which we call the \emph{projection};
\item $z: M \to E$, a section of $q$ which we call the \emph{zero section};
\item $\sigma: E \times_M E \to E$, the \emph{addition}, which plays the role of fibrewise addition of vectors; and
\item $\lambda: E \to TE$, the \emph{vertical lift}, which is less familiar from the theory of vector bundles, but not hard to define.
\end{itemize}
These four morphisms are subject to a number of axioms; see~\cite[2.3]{cockett/cruttwell:2018} for a full list. MacAdam shows in~\cite{macadam:2021} that the Cockett-Cruttwell definition precisely describes smooth vector bundles in the tangent category of smooth manifolds. There is also a notion of \emph{linear map} between differential bundles~\cite[2.3]{cockett/cruttwell:2018} (pairs of morphisms that commute with all the structure maps), and in the tangent category $\Mfld$ these are the usual thing: maps of vector bundles which are linear on each fibre.

Our goal in this note is to provide a new description of differential bundles in an arbitrary tangent category (Corollary~\ref{cor}). That description is inspired by similar results of MacAdam, specifically~\cite[Proposition 6]{macadam:2021}, but our characterization is simpler than his. There are two main parts to our work. Firstly, we have the following result.

\begin{intro-prop}
Let $q: E \to M$ be a morphism in a tangent category $\bcat{X}$, and let $z: M \to E$ be a section of $q$. Then (subject to the existence and preservation by $T$ of certain limits), there is a differential bundle in $\bcat{X}$ whose underlying projection map takes the form
\[ q': M \times_{TM} TE \times_{E} M \to M. \]
\end{intro-prop}

The source object of the map $q'$ is the limit of the following diagram in $\bcat{X}$:
\[\begin{tikzcd}
	M && TE && M \\
	& TM && E
	\arrow["{0_M}"', from=1-1, to=2-2]
	\arrow["{T(q)}"', from=1-3, to=2-2]
	\arrow["{p_E}", from=1-3, to=2-4]
	\arrow["z", from=1-5, to=2-4]
\end{tikzcd}\]
with $q'$ given by the canonical projection from this limit to either of the two copies of $M$ in the diagram. (As part of the proof of this proposition, we show that those two projections are necessarily equal.)

Note that the existence of the differential bundle in Proposition~\ref{prop:lim} does not rely at all on $q$ and $z$ themselves being part of the structure of a differential bundle; any morphism and section can be used. 

Here is our second result, which completes the classification of differential bundles.

\begin{intro-thm}
Let $q: E \to M$, $z: M \to E$, and $\lambda: E \to TE$ be structure maps for a differential bundle in a tangent category $\bcat{X}$. Then there is a linear isomorphism of differential bundles over $M$ given by the map
\[  \pb{q,\lambda,q}: E \stackrel{\isom}{\longrightarrow} M \times_{TM} TE \times_{E} M. \]
Thus the differential bundles in $\bcat{X}$ are, up to isomorphism, precisely those described in Proposition~\ref{prop:lim}.
\end{intro-thm}

One consequence of this theorem is that, up to isomorphism, a differential bundle is determined by its projection map and zero section. The vertical lift then identifies the differential bundle within its isomorphism class. It also follows that the addition morphism $\sigma$ is completely determined by the other parts of the structure; that fact is also established by MacAdam in~\cite[Lemma 5]{macadam:2021}.

Here is a preview of the proofs of our main results. We construct the differential bundle in Proposition~\ref{prop:lim} as the limit of a diagram in the category $\DBun(\bcat{X})_{\mathrm{lin}}$ of differential bundles and linear maps of the form:
\begin{equation} \label{eq:q'} \tag{*} \begin{tikzcd}
	{\un{c}M} && {\un{T}E} && {\un{c}M} \\
	& {\un{T}M} && {\un{c}E}
	\arrow[from=1-1, to=2-2]
	\arrow[from=1-3, to=2-2]
	\arrow[from=1-3, to=2-4]
	\arrow[from=1-5, to=2-4]
\end{tikzcd} \end{equation}
where $\un{T}E$ and $\un{T}M$ are the tangent bundles on $E$ and $M$, respectively, and $\un{c}E$ and $\un{c}M$ are the trivial bundles (whose underlying projection maps and zero section are both identity morphisms). See Lemma~\ref{lem:key} for a precise construction of this diagram.

In order to form the double pullback (*) we need to know how to take limits in the category $\DBun(\bcat{X})_{\mathrm{lin}}$. In Lemma~\ref{lem:lim} we give a formula for those limits and provide conditions under which they exist. Those conditions require that each underlying diagram in $\bcat{X}$ has a limit, and that the limit is preserved by the application of the iterated tangent bundle functors $T^n$ for $n \geq 1$.

In many cases, for example if $T$ is already known to preserve limits, or at least certain kinds of limits, those conditions are automatic. For example, in $\Mfld$, the tangent bundle functor preserves pullbacks along submersions of smooth manifolds, and that fact is sufficient to deduce that the differential bundle $q'$ always exists. In tangent categories arising from models of SDG, or in the standard tangent structures on commutative rings, the functor $T$ preserves all limits, and there is nothing extra to check. In general, however, the extra conditions need to be considered separately. 

The proof of Theorem~\ref{thm:main} amounts to a computation that when $q$ and $z$ are already part of a differential bundle $\un{E}$, then those additional conditions hold and $\un{E}$ itself provides a limit for the diagram (*).

We are grateful to a referee for pointing out the following alternative approach to our results. A morphism $q: E \to M$ in a tangent category $\bcat{X}$ is an object of the slice category $\bcat{X}_{/M}$, and a section $z$ of $q$ is a (generalized) point of that object, i.e.\ a morphism in $\bcat{X}_{/M}$ from the terminal object, which is the identity morphism on $M$, to $q$. Assuming the existence (and preservation by $T$) of sufficient pullbacks, Cockett and Cruttwell show in~\cite[5.7]{cockett/cruttwell:2018} that $\bcat{X}_{/M}$ inherits a tangent structure from that on $\bcat{X}$. The same authors also show in~\cite[4.15]{cockett/cruttwell:2014} that the tangent \emph{space} at a generalized point in any object in a tangent category forms a differential object. Combining these results, we deduce that the tangent space $T_z(q)$, when it exists, is a differential object in the slice tangent category $\bcat{X}_{/M}$. Working through the construction of that tangent category we see that $T_z(q)$ is precisely the projection map
\[ M \times_{TM} TE \times_E M \to M \]
which underlies the differential bundle in Proposition~\ref{prop:lim}. By~\cite[5.12]{cockett/cruttwell:2018} that differential object in $\bcat{X}_{/M}$ is a differential bundle in $\bcat{X}$. Conversely, any differential bundle $q: E \to M$ is a differential object in the slice tangent category, and so is isomorphic to its own tangent space $T_z(q)$. That argument recovers Theorem~\ref{thm:main}, though only by assuming the existence of the slice tangent structure which is not needed in our proof.

One might ask if our perspective also provides a new way to understand linear maps between differential bundles. We can construct linear maps via this method, but unfortunately it does not provide a complete classification, as we show in \S\ref{sec:linear}.

In {\S}\ref{sec:ex} we turn to examples. We use Theorem~\ref{thm:main} to identify the differential bundles in a variety of tangent categories. In almost all of those cases, this recovers previously known results but by a much shorter argument. We also identify differential objects by looking at differential bundles over a terminal object, recovering Cockett and Cruttwell's result, mentioned above that differential objects in a tangent category are, up to isomorphism, precisely the tangent \emph{spaces}, or fibres of the tangent bundle over a point~\cite[4.15]{cockett/cruttwell:2014}.

Finally, while we do not consider here the tangent $\infty$-categories of~\cite{bauer/burke/ching:2023}, we expect that the main results of this paper, and largely also their proofs, will hold equally well in the $\infty$-categorical setting. In fact, the motivation for this paper was the goal of classifying differential bundles in the Goodwillie tangent structure on the $\infty$-category of (differentiable) $\infty$-categories~\cite{bauer/burke/ching:2023}, and we hope to return to that goal in future work.

\subsection*{Notation}
Throughout this paper we follow the notation of Cockett and Cruttwell~\cite{cockett/cruttwell:2018} as closely as possible, with the exception that we denote the zero section of a bundle by $z$ instead of $\zeta$. We write composition of morphisms in a category in diagrammatic order, so that $fg$ denotes $f$ followed by $g$. 

For a differential bundle $\un{E} = (q,z,\sigma,\lambda)$ we write $E_k$ for the wide pullback of $k$ copies of the projection map $q: E \to M$, i.e.\
\[ E_k := E \times_M \dots \times_M E \]
with $k$ factors of $E$. In particular we have $E_1 = E$ and $E_0 = M$.

We often express a morphism of the form $f: A \to B \times_C D$ by writing
\[ f = \pb{f_1,f_2} \]
where $f_1: A \to B$ and $f_2: A \to D$ are the morphisms that uniquely determine $f$. A similar notation is used for maps into other types of limit when they are clearly determined by a sequence of morphisms in the same way. 

\subsection*{Acknowledgements}
This paper is inspired by Ben MacAdam's work~\cite{macadam:2021}, and it came to be thanks also to conversations with Kaya Arro, Geoff Cruttwell, and Florian Schwarz. Particular thanks for JS Lemay for pointing out a mistake in the original version of Example~\ref{ex:CDC}, and to a referee for many suggestions and corrections. I'm very grateful to the organizers of the International Category Theory Conference (CT2024) in Santiago de Compostela, where these results were first presented.

\section{Differential bundles in a tangent category} \label{sec:bun}

We refer the reader to~\cite[2.1, 2.3]{cockett/cruttwell:2018} for full definitions of tangent category, of differential bundle, and of linear map between differential bundles. We will make repeated use of the axioms listed there.

Let $\bcat{X}$ be a tangent category,  and let $q: E \to M$ be a morphism in $\bcat{X}$ with section $z: M \to E$, i.e.\ such that $zq = 1_M$.

Let $\un{c}E := (1_E,1_E,1_E,0_E)$ denote the trivial differential bundle on $E$~\cite[2.4(i)]{cockett/cruttwell:2018}. Let $\un{T}E := (p_E,0_E,+_E,\ell_E)$ denote the tangent bundle on $E$~\cite[2.4(ii)]{cockett/cruttwell:2018}. From the morphisms $q$ and $z$, we construct a diagram of differential bundles in $\bcat{X}$ as follows.

\begin{lemma} \label{lem:key}
There is a diagram of differential bundles (and linear maps) of the form
\[\begin{tikzcd}
	{\un{c}M} && {\un{T}E} && {\un{c}M} \\
	\\
	& {\un{T}M} && {\un{c}E}
	\arrow["{(0_M,1_M)}"', from=1-1, to=3-2]
	\arrow["{(T(q),q)}"{description}, from=1-3, to=3-2]
	\arrow["{(p_E,1_E)}"{description}, from=1-3, to=3-4]
	\arrow["{(z,z)}", from=1-5, to=3-4]
\end{tikzcd}\]
where we have denoted a linear bundle map by the pair consisting of the morphism between total objects followed by the morphism between base objects.
\end{lemma}
\begin{proof}
We have to show that each of the maps in the diagram is indeed a linear bundle map, i.e. commutes with the projection and vertical lift. (By~\cite[2.16]{cockett/cruttwell:2018}, each map then also commutes with the zero sections and addition.) We take each in turn.

For $(0_M,1_M)$, we have the following commutative diagrams (each of which is a tangent category axiom):
\[\begin{tikzcd}
	M & TM & M && TM \\
	M & M & TM && {T^2M}
	\arrow["{0_M}", from=1-1, to=1-2]
	\arrow["{1_M}"', from=1-1, to=2-1]
	\arrow["{p_M}", from=1-2, to=2-2]
	\arrow["{0_M}", from=1-3, to=1-5]
	\arrow["{0_M}"', from=1-3, to=2-3]
	\arrow["{\ell_M}", from=1-5, to=2-5]
	\arrow["{1_M}"', from=2-1, to=2-2]
	\arrow["{T(0_M)}"', from=2-3, to=2-5]
\end{tikzcd}\]
where we note that the vertical lift for the trivial bundle on $M$ is $0_M: M \to TM$ and on the tangent bundle on $M$ is $\ell_M: TM \to T^2M$.

For $(z,z)$, the relevant diagrams are the following:
\[\begin{tikzcd}
	M & E & M && E \\
	M & E & TM && TE
	\arrow["z", from=1-1, to=1-2]
	\arrow["{1_M}"', from=1-1, to=2-1]
	\arrow["{1_E}", from=1-2, to=2-2]
	\arrow["z", from=1-3, to=1-5]
	\arrow["{0_M}"', from=1-3, to=2-3]
	\arrow["{0_E}", from=1-5, to=2-5]
	\arrow["z"', from=2-1, to=2-2]
	\arrow["{T(z)}"', from=2-3, to=2-5]
\end{tikzcd}\]
where the second is a naturality square for $0$.

For $(p_E,1_E)$, the diagrams are:
\[\begin{tikzcd}
	TE & E & TE && E \\
	E & E & {T^2E} && TE
	\arrow["{p_E}", from=1-1, to=1-2]
	\arrow["{p_E}"', from=1-1, to=2-1]
	\arrow["{1_E}", from=1-2, to=2-2]
	\arrow["{p_E}", from=1-3, to=1-5]
	\arrow["{\ell_E}"', from=1-3, to=2-3]
	\arrow["{0_E}", from=1-5, to=2-5]
	\arrow["{1_E}"', from=2-1, to=2-2]
	\arrow["{T(p_E)}"', from=2-3, to=2-5]
\end{tikzcd}\]
where commutativity of the second diagram is one of the tangent category axioms. 

Finally, for $(T(q),q)$, the diagrams are
\[\begin{tikzcd}
	TE & TM & TE && TM \\
	E & M & {T^2E} && {T^2M}
	\arrow["{T(q)}", from=1-1, to=1-2]
	\arrow["{p_E}"', from=1-1, to=2-1]
	\arrow["{p_M}", from=1-2, to=2-2]
	\arrow["{T(q)}", from=1-3, to=1-5]
	\arrow["{\ell_E}"', from=1-3, to=2-3]
	\arrow["{\ell_M}", from=1-5, to=2-5]
	\arrow["q"', from=2-1, to=2-2]
	\arrow["{T^2(q)}"', from=2-3, to=2-5]
\end{tikzcd}\]
which are naturality squares for $p$ and $\ell$, respectively.
\end{proof}

Our first goal is to find conditions under which the diagram in Lemma~\ref{lem:key} admits a limit. In the following result, we provide sufficient conditions for a limit of an arbitrary diagram of differential bundles (and linear maps) to exist.

\begin{lemma} \label{lem:lim}
Let $\un{E}^\bullet: \bcat{I} \to \mathrm{DBun}(\bcat{X})_{\mathrm{lin}}$ be a diagram in the category of differential bundles in $\bcat{X}$ and linear maps. Equivalently, $\un{E}^\bullet$ is a differential bundle in the functor tangent category $\Fun(\bcat{I},\bcat{X})$ (with the pointwise tangent structure of~\cite[2.2(vii)]{cockett/cruttwell:2018}), i.e. consists of a natural transformation
\[ q^\bullet: E_1^\bullet \to E_0^\bullet \]
along with natural transformations
\[ z^\bullet: E_0^\bullet \to E_1^\bullet, \quad \sigma^\bullet: E_2^\bullet := E_1^\bullet \times_{E_0^\bullet} E_1^\bullet \to E_1^\bullet, \quad \lambda: E_1^\bullet \to TE_1^\bullet \]
which satisfy the Cockett-Cruttwell axioms. In particular, for each $k \geq 1$, the wide pullback $E_k^\bullet$ of $k$ copies of $q^\bullet$ exists, and is preserved by $T^n$ for all $n \geq 1$.

\begin{itemize}
\item[(*)] Suppose that for each $k \geq 0$, the diagram $E_k^\bullet: \bcat{I} \to \bcat{X}$ admits a limit in $\bcat{X}$, which is preserved by $T^n$ for all $n \geq 1$, i.e. the canonical map
\[T^n( \lim E_k^\bullet) \longrightarrow \lim T^n(E_k^\bullet) \]
is an isomorphism in $\bcat{X}$.
\end{itemize}

Then the diagram $\un{E}^\bullet$ has a limit $\un{E}$ in $\mathrm{DBun}(\bcat{X})_{\mathrm{lin}}$ with
\[ E_k = \lim E_k^\bullet \]
and structure maps
\[ q = \lim q^\bullet, \quad z = \lim z^\bullet, \quad \sigma = \lim \sigma^\bullet \]
and $\lambda: E_1 \to TE_1$ equal to the composite
\[\begin{tikzcd}
	{E_1 = \lim E_1^\bullet} && {\lim TE_1^\bullet \isom T(\lim E_1^\bullet) = TE_1}
	\arrow["{\lim \lambda^\bullet}", from=1-1, to=1-3]
\end{tikzcd}\]
\end{lemma}
\begin{proof}
Note that a similar result (stated only for products of differential bundles over a fixed base) appears in~\cite[4.4]{lucyshyn-wright:2018}.

Let $q = \lim q^\bullet: E_1 \to E_0$. We first show that for each $k \geq 0$, the wide pullback of $k$ copies of $q$ exists and is preserved by $T^n$ for all $n \geq 1$. We illustrate the calculation for $k = 2$; the same works for larger $k$. We have a sequence of canonical isomorphisms (for any $n \geq 0$):
\begin{equation} \label{eq:Tnlim} \begin{split} T^n(E_2) &\isom \lim T^n(E_2^\bullet) \\ \\
	&\isom  \lim T^n(E_1^\bullet \times_{E_0^\bullet} E_1^\bullet) \\ \\
	&\isom \lim \left[ T^n(E_1^\bullet) \times_{T^n(E_0^\bullet)} T^n(E_1^\bullet) \right] \\ \\
	&\isom \left( \lim T^n(E_1^\bullet) \right) \times_{\left( \lim T^n(E_0^\bullet) \right)} \left( \lim T^n(E_1^\bullet) \right) \\ \\
	&\isom T^n \left( \lim E_1^\bullet \right) \times_{T^n\left( \lim E_0^\bullet \right)} T^n \left( \lim E_1^\bullet \right) \\ \\
	&= T^n(E_1) \times_{T^n(E_0)} T^n(E_1)
\end{split} \end{equation}
where we have used, respectively, the hypothesis (*), the definition of $E_2^\bullet$ as the pullback of copies of $q^\bullet$, the fact that $T^n$ preserves those pullbacks (since $\un{E}^\bullet$ is a diagram of differential bundles), limits commute with limits, hypothesis (*) again, and the definitions of $E_1$ and $E_0$.

Now let the remainder of the structure maps on $\un{E}$ be given as in the statement of the lemma. To check that these maps make $\un{E}$ into a differential bundle in $\bcat{X}$, we can check the remaining axioms from~\cite[2.3]{cockett/cruttwell:2018}. The majority of those axioms say that some diagram involving the structure maps commutes, and those follow from the commutativity of the corresponding diagrams for the individual differential bundles in the diagram $\un{E}^\bullet$ (together with, for those diagrams that involve the application of $T$, the naturality of the various parts of the tangent structure on $\bcat{X}$).

The remaining axiom is the universality of the vertical lift, which requires that a certain diagram in $\bcat{X}$ is a pullback. That follows from the corresponding axiom for the individual differential bundles in $\un{E}^\bullet$ because pullbacks commute with other limits.

So, $\un{E}$ is a differential bundle. By construction it comes with maps to each of the bundles $\un{E}^i$, which commute both with the maps in the diagram $
\un{E}^\bullet$ and the structure maps, so $\un{E}$ forms a cone over $\un{E}^\bullet$. Suppose we have another cone formed by maps of differential bundles
\[ f^\bullet: \un{X} \to \un{E}^\bullet. \]
In particular, we have underlying cones $f_0^\bullet : X_0 \to E_0^\bullet$ and $f_1:^\bullet  X_1 \to E_1^\bullet$, which determine unique morphisms $f_0: X_0 \to E_0$ and $f_1: X_1 \to E_1$, respectively, which by the universal properties of the limits commute with the structure maps, and so uniquely determine a linear map $f: \un{X} \to \un{E}$. Thus $\un{E}$ is the limit of $\un{E}^\bullet$ in the category $\mathrm{DBun}(\bcat{X})_{\mathrm{lin}}$.
\end{proof}

Applying Lemma~\ref{lem:lim} to the diagram of bundles in Lemma~\ref{lem:key}, we can now obtain our first main result about the existence of differential bundles constructed from a given morphism $q: E \to M$ with section $z: M \to E$.

\begin{proposition} \label{prop:lim}
Let $q: E \to M$ be a morphism in a tangent category $\bcat{X}$ with section $z: M \to E$, such that for all $k \geq 2$, the following double pullback exists
\[ M \times_{T_kM} T_kE \times_{E} M \]
and is preserved by $T^n$ for all $n \geq 1$, where the maps involved in the pullbacks are $(0_k)_M$, $T_k(q)$, $(0_k)_M$, and $z$. (We are writing $p_k: T_k \to I_{\bcat{X}}$ and $0_k: I_{\bcat{X}} \to T_k$ for the iterated projection and zero section maps associated to the tangent structure.)

Then there is a differential bundle with structure maps:
\begin{itemize} \itemsep=5pt
\item $q': M \times_{TM} TE \times_{E} M \to M$ given by projection onto either factor of $M$;
\item $z' = \pb{1_M, z 0_E, 1_M}: M \to M \times_{TM} TE \times_{E} M$;
\item $\sigma'$ given by an isomorphism
\[ (M \times_{TM} TE \times_E M) \times_M (M \times_{TM} TE \times_E M) \isom (M \times_{T_2M} T_2E \times_E M) \]
followed by the map
\[\begin{tikzcd}
	{M \times_{T_2M} T_2E \times_E M} &&& {M \times_{TM} TE  \times_E M;}
	\arrow["{1_M \times_{(+_M)}(+_E) \times_{1_E} 1_M}", from=1-1, to=1-4]
\end{tikzcd}\]
\item $\lambda'$ given by
\[ \begin{tikzcd}
	{M \times_{TM} TE \times_{E} M} &&& {TM \times_{T^2M} T^2E \times_{TE} TM \isom T(M \times_{TM} TE \times _E M).}
	\arrow["{0_M \times_{\ell_M} \ell_E \times_{0_E} 0_M}", from=1-1, to=1-4]
\end{tikzcd} \]
\end{itemize}
\end{proposition}
\begin{proof}
The given hypotheses state precisely that the conditions of Lemma~\ref{lem:lim} are satisfied for the diagram of differential bundles in Lemma~\ref{lem:key}. Therefore a limit differential bundle exists, which we denote $\un{V}$. We will show that this bundle takes the form stated.

The object $V_1$ is, by definition, the limit of the diagram
\[\begin{tikzcd}
	M && TE && M \\
	& TM && E
	\arrow["{0_M}"', from=1-1, to=2-2]
	\arrow["{T(q)}"', from=1-3, to=2-2]
	\arrow["{p_E}", from=1-3, to=2-4]
	\arrow["z", from=1-5, to=2-4]
\end{tikzcd}.\]
The object $V_0$ is the limit of the diagram
\[\begin{tikzcd}
	M && E && M \\
	& M && E
	\arrow["{1_M}"', from=1-1, to=2-2]
	\arrow["q"', from=1-3, to=2-2]
	\arrow["{1_E}", from=1-3, to=2-4]
	\arrow["z", from=1-5, to=2-4]
\end{tikzcd}.\]
Because $zq = 1_M$, there is a cone over this diagram with vertex $M$ consisting of the maps $\pb{1_M,z,1_M}$. Suppose the maps $f:X \to M$, $g: X \to E$, $h: X \to M$ form another cone over this diagram. Then
\[ f = gq, \quad g = hz \]
and so
\[ f = hzq = h. \]
It follows that this second cone factors uniquely through the first via the map $h$, and hence that the object $V_0$ is isomorphic to $M$ via either of the two canonical projection maps $V_0 \to M$, which are necessarily equal.

The projection map $q': V_1 \to V_0$ is the map between the limits of the above diagrams induced by the projection maps $p_E$, $p_M$, and the identity morphisms on $E$ and $M$. Combining that map with the isomorphisms we have already described, we identify $q'$ with either of the two projection maps
\[ q': M \times_{TM} TE \times_E M \to M \]
which must be equal. The structure maps $z': V_0 \to V_1$, $\sigma': V_2 \to V_1$ and $\lambda': V_1 \to T(V_1)$ are then easily calculated from the corresponding maps in the diagram of differential bundles in Lemma~\ref{lem:key}.
\end{proof}

Having constructed a collection of differential bundles in $\bcat{X}$, our second goal is to show that any differential bundle is isomorphic to one of those. So now we suppose that $q: E \to M$ and $z: M \to E$ are already the projection and zero section for a differential bundle $\un{E}$ with vertical lift: $\lambda: E \to TE$.

\begin{lemma} \label{lem:cone}
There is a cone with vertex $\un{E}$ over the diagram of differential bundles of Lemma~\ref{lem:key} given by
\[\begin{tikzcd}
	&& {\un{E}} \\
	{\un{c}M} && {\un{T}E} && {\un{c}M} \\
	\\
	& {\un{T}M} && {\un{c}E}
	\arrow["{(q,1_M)}"', from=1-3, to=2-1]
	\arrow["{(\lambda,z)}", from=1-3, to=2-3]
	\arrow["{(q,1_M)}", from=1-3, to=2-5]
	\arrow["{(0_M,1_M)}"', from=2-1, to=4-2]
	\arrow["{(T(q),q)}"', from=2-3, to=4-2]
	\arrow["{(p_E,1_E)}", from=2-3, to=4-4]
	\arrow["{(z,z)}", from=2-5, to=4-4]
\end{tikzcd}\]
\end{lemma}
\begin{proof}
We first show that $(q,1_M): \un{E} \to \un{c}M$ is a linear map. The diagrams which need to commute are
\[\begin{tikzcd}
	E & M & E && M \\
	M & M & TE && TM
	\arrow["q", from=1-1, to=1-2]
	\arrow["q"', from=1-1, to=2-1]
	\arrow["{1_M}", from=1-2, to=2-2]
	\arrow["q", from=1-3, to=1-5]
	\arrow["\lambda"', from=1-3, to=2-3]
	\arrow["{0_M}", from=1-5, to=2-5]
	\arrow["{1_M}"', from=2-1, to=2-2]
	\arrow["{T(q)}"', from=2-3, to=2-5]
\end{tikzcd}\]
with the second one commuting by one of the differential bundle axioms. Next we show that $(\lambda,z): \un{E} \to \un{T}E$ is a linear map. The relevant diagrams are
\[\begin{tikzcd}
	E & TE & E && TE \\
	M & E & TE && {T^2E}
	\arrow["\lambda", from=1-1, to=1-2]
	\arrow["q"', from=1-1, to=2-1]
	\arrow["{p_E}", from=1-2, to=2-2]
	\arrow["\lambda", from=1-3, to=1-5]
	\arrow["\lambda"', from=1-3, to=2-3]
	\arrow["{\ell_E}", from=1-5, to=2-5]
	\arrow["z"', from=2-1, to=2-2]
	\arrow["{T(\lambda)}"', from=2-3, to=2-5]
\end{tikzcd}\]
which both comute by differential bundle axioms.

Finally, we must show that the diagram of bundles commutes, i.e. the diagram of total objects commutes and the diagram of base objects commutes. For the total objects, the diagram is
\[\begin{tikzcd}
	&& E \\
	M && TE && M \\
	\\
	& TM && E
	\arrow["q"', from=1-3, to=2-1]
	\arrow["\lambda", from=1-3, to=2-3]
	\arrow["q", from=1-3, to=2-5]
	\arrow["{0_M}"', from=2-1, to=4-2]
	\arrow["{T(q)}"', from=2-3, to=4-2]
	\arrow["{p_E}", from=2-3, to=4-4]
	\arrow["z", from=2-5, to=4-4]
\end{tikzcd}\]
and the two squares therein commute by differential bundle axioms. For the base objects, the diagram is
\[\begin{tikzcd}
	&& M \\
	M && E && M \\
	\\
	& M && E
	\arrow["{1_M}"', from=1-3, to=2-1]
	\arrow["z", from=1-3, to=2-3]
	\arrow["{1_M}", from=1-3, to=2-5]
	\arrow["{1_M}"', from=2-1, to=4-2]
	\arrow["q"', from=2-3, to=4-2]
	\arrow["{1_E}", from=2-3, to=4-4]
	\arrow["z", from=2-5, to=4-4]
\end{tikzcd}\]
which commutes because $zq = 1_M$.
\end{proof}

The main result of this paper is that the diagram of bundles in Lemma~\ref{lem:cone} is actually a limit cone for the diagram in Lemma~\ref{lem:key}. That result can be stated as follows.

\begin{theorem} \label{thm:main}
Let $(q,z,\sigma,\lambda)$ be a differential bundle in a tangent category $\bcat{X}$. Then the hypotheses of Proposition~\ref{prop:lim} are satisfied by the morphisms $q$ and $z$, and there is a linear isomorphism of differential bundles over $M$ given by
\[\begin{tikzcd}
	E && {M \times_{TM} TE \times_E M} \\
	& M
	\arrow["{\pb{q,\lambda,q}}", "\isom"',from=1-1, to=1-3]
	\arrow["q"', from=1-1, to=2-2]
	\arrow["{q'}", from=1-3, to=2-2]
\end{tikzcd}\]
Thus, every differential bundle in $\bcat{X}$ is, up to isomorphism, of the form constructed in Proposition~\ref{prop:lim}.
\end{theorem}

\begin{proof}
The diagram of differential bundles in Lemma~\ref{lem:cone} determines diagrams in $\bcat{X}$ of the form
\begin{equation} \label{eq:diag}
\begin{tikzcd}
	&& {E_k} \\
	M && {T_kE} && M \\
	\\
	& {T_kM} && E
	\arrow["{q_k}"', from=1-3, to=2-1]
	\arrow["{\lambda_k}", from=1-3, to=2-3]
	\arrow["{q_k}", from=1-3, to=2-5]
	\arrow["{(0_k)_M}"', from=2-1, to=4-2]
	\arrow["{T_k(q)}"', from=2-3, to=4-2]
	\arrow["{(p_k)_E}", from=2-3, to=4-4]
	\arrow["z", from=2-5, to=4-4]
\end{tikzcd}
\end{equation}
where $q_k: E_k \to M$ is the wide pullback of $k$ copies of $q$, and $\lambda_k: E_k \to T_kE$ is the map
\[ E_k = E \times_M \dots \times_M E \to TE \times_{E} \dots \times_{E} TE = T_kE \]
induced by $\lambda: E \to TE$ and $z: M \to E$.

We claim that each of these diagrams is a limit. In the case $k = 0$, the diagram takes the form 
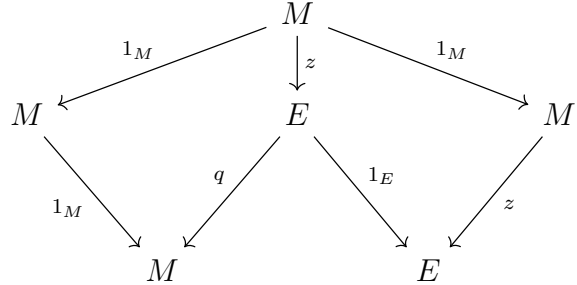
\[\begin{tikzcd}
	&& M \\
	M && E && M \\
	\\
	& M && E
	\arrow["{1_M}"', from=1-3, to=2-1]
	\arrow["z", from=1-3, to=2-3]
	\arrow["{1_M}", from=1-3, to=2-5]
	\arrow["{1_M}"', from=2-1, to=4-2]
	\arrow["q"', from=2-3, to=4-2]
	\arrow["{1_E}", from=2-3, to=4-4]
	\arrow["z", from=2-5, to=4-4]
\end{tikzcd}\]
which, we have already noted, is a limit by direct calculation.

Now consider the case $k = 1$, which is the heart of the proof of this theorem. We will reorient the necessary diagram as
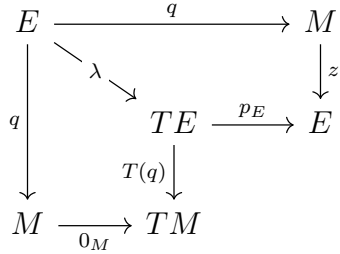
\[\begin{tikzcd}
	E && M \\
	& TE & E \\
	M & TM &
	\arrow["q", from=1-1, to=1-3]
	\arrow["\lambda"{description}, from=1-1, to=2-2]
	\arrow["q"', from=1-1, to=3-1]
	\arrow["z", from=1-3, to=2-3]
	\arrow["{p_E}", from=2-2, to=2-3]
	\arrow["{T(q)}"', from=2-2, to=3-2]
	\arrow["{0_M}"', from=3-1, to=3-2]
\end{tikzcd}\]

That diagram factors as follows
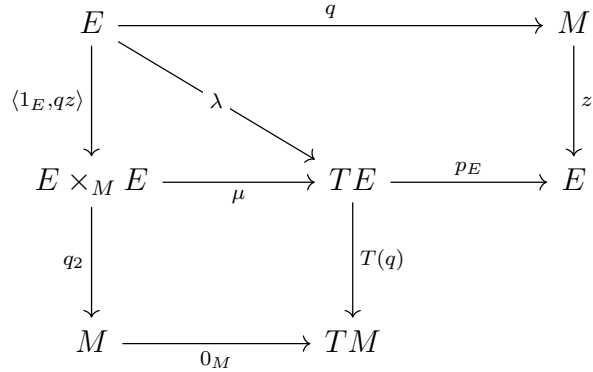
\[\begin{tikzcd}
	E &&&& M \\
	\\
	{E \times_M E} && TE && E \\
	\\
	M && TM
	\arrow["q", from=1-1, to=1-5]
	\arrow["{\pb{1_E,qz}}"', from=1-1, to=3-1]
	\arrow["\lambda"{description}, from=1-1, to=3-3]
	\arrow["z", from=1-5, to=3-5]
	\arrow["\mu"', from=3-1, to=3-3]
	\arrow["{q_2}"', from=3-1, to=5-1]
	\arrow["{p_E}", from=3-3, to=3-5]
	\arrow["{T(q)}", from=3-3, to=5-3]
	\arrow["{0_M}"', from=5-1, to=5-3]
\end{tikzcd}\]
where $\mu = \pb{\pi_1\lambda,\pi_2 0_E}T(\sigma)$ is the map appearing in the vertical lift axiom for a differential bundle. The top-left vertical map is well-defined because $1_Eq = qzq$, and this also shows that it factors the map $q: E \to M$ from the previous diagram. The bottom-left square commutes because it is the pullback square in the vertical lift axiom for a differential bundle, and the triangle commutes by~\cite[2.9]{cockett/cruttwell:2018}.

It is now sufficient to show that the diagram above is formed from pullback squares as shown here:
\begin{equation} \label{eq:proof} 
\begin{tikzcd}
	E &&&& M \\
	\\
	{E \times_M E} && TE && E \\
	M && TM
	\arrow["q", from=1-1, to=1-5]
	\arrow["{\pb{1_E,qz}}"', from=1-1, to=3-1]
	\arrow["z", from=1-5, to=3-5]
	\arrow[""{name=0, anchor=center, inner sep=0}, "\mu", from=3-1, to=3-3]
	\arrow["{q_2}"', from=3-1, to=4-1]
	\arrow["\lrcorner"{anchor=center, pos=0.125}, draw=none, from=3-1, to=4-3]
	\arrow["{p_E}", from=3-3, to=3-5]
	\arrow["{T(q)}", from=3-3, to=4-3]
	\arrow["{0_M}"', from=4-1, to=4-3]
	\arrow["\lrcorner"{anchor=center, pos=0.125}, draw=none, from=1-1, to=0]
\end{tikzcd} \end{equation}
The bottom square is precisely the pullback of the vertical lift axiom for a differential bundle. By~\cite[2.9]{cockett/cruttwell:2018}, the middle horizontal composite $\mu p_E$ is equal to projection onto the second factor of $E$, and it is then a direct calculation that the top rectangle is also a pullback. This completes the proof that (\ref{eq:diag}) is a limit diagram for $k = 1$.

The diagram (\ref{eq:diag}) for $k \geq 2$ is equal to the objectwise wide pullback of $k$ copies of the diagram for $k = 1$ over the diagram for $k = 0$. Since wide pullbacks commute with limits, it follows that the diagram for $k \geq 2$ is also a limit.

Next we have to show that the limit diagrams (\ref{eq:diag}) are preserved by $T^n$ for all $n \geq 1$. We start with $k = 0$ for which we have to show that
\[\begin{tikzcd}
	&& {T^n(M)} \\
	{T^n(M)} && {T^n(E)} && {T^n(M)} \\
	& {T^n(M)} && {T^n(E)}
	\arrow["{1_{T^n(M)}}"', from=1-3, to=2-1]
	\arrow["{T^n(z)}", from=1-3, to=2-3]
	\arrow["{1_{T^n(M)}}", from=1-3, to=2-5]
	\arrow["{1_{T^n(M)}}"', from=2-1, to=3-2]
	\arrow["{T^n(q)}"', from=2-3, to=3-2]
	\arrow["{1_{T^n(E)}}", from=2-3, to=3-4]
	\arrow["{T^n(z)}", from=2-5, to=3-4]
\end{tikzcd}\]
is a limit diagram; again, that is a direct calculation.

For $k = 1$ it is sufficient, by the same argument as earlier in this proof, to observe that in the following diagram:
\[\begin{tikzcd}
	{T^n(E)} &&&& {T^n(M)} \\
	\\
	{T^n(E \times_M E)} && {T^n(TE)} && {T^n(E)} \\
	\\
	{T^n(M)} && {T^n(TM)}
	\arrow["{T^n(q)}", from=1-1, to=1-5]
	\arrow["{T^n(\pb{1_E,qz})}"', from=1-1, to=3-1]
	\arrow["{T^n(z)}", from=1-5, to=3-5]
	\arrow["{T^n(\mu)}", from=3-1, to=3-3]
	\arrow["{T^n(q_2)}"', from=3-1, to=5-1]
	\arrow["{T^n(p_E)}", from=3-3, to=3-5]
	\arrow["{T^n(T(q))}", from=3-3, to=5-3]
	\arrow["{T^n(0_M)}"', from=5-1, to=5-3]
\end{tikzcd}\]
the two squares are pullbacks. The bottom square is a pullback because $T^n$ preserves the vertical lift pullback (differential bundle axiom). Since $T^n$ also preserves the pullback $E \times_M E$ (another differential bundle axiom), and using~\cite[2.9]{cockett/cruttwell:2018} again, the top square can be rewritten up to isomorphism as
\[\begin{tikzcd}
	{T^n(E)} && {T^n(M)} \\
	\\
	{T^n(E) \times_{T^n(M)} T^n(E)} && {T^n(E)}
	\arrow["{T^n(q)}", from=1-1, to=1-3]
	\arrow["{\pb{1_{T^n(E)},T^n(q)T^n(z)}}"', from=1-1, to=3-1]
	\arrow["{T^n(z)}", from=1-3, to=3-3]
	\arrow["{\pi_2}"', from=3-1, to=3-3]
\end{tikzcd}\]
which can be directly calculated to be a pullback. That completes the proof that the limit diagram (\ref{eq:diag}) is preserved by $T^n$ when $k = 1$.

The argument for $k \geq 2$ is now essentially that of Lemma~\ref{lem:lim} in reverse. We give it here for $k = 2$; the same works for larger $k$.

Writing $E^\bullet_2$ for the diagram in (\ref{eq:diag}) whose limit is $E_2$, we have
\[ \begin{split} T^n(E_2) &\isom T^n(E_1) \times_{T^n(E_0)} T^n(E_1) \\ \\
	&\isom (\lim T^n(E_1^\bullet)) \times_{(\lim T^n(E_0^\bullet))} (\lim T^n(E_1^\bullet)) \\ \\
	&\isom \lim \left[ T^n(E_1^\bullet) \times_{T^n(E_0^\bullet)} T^n(E_1^\bullet) \right] \\ \\
	&\isom \lim T^n(E_2^\bullet)
\end{split} \]
using: the differential bundle axiom for $E$, what we have already shown in the cases $k = 0,1$, limits commuting with limits, and the differential bundle axioms applied to the consitutent bundles in the diagram $\un{E}^\bullet$. This completes the proof that $T^n$ preserves all the limit diagrams in (\ref{eq:diag}). Formulating each of those in terms of the double pullback, we get the hypotheses of Proposition~\ref{prop:lim}.

The diagram of differential bundles in Lemma~\ref{lem:cone} now induces a linear map of differential bundles
\[ \pb{(q,1_M),(\lambda,z),(q,1_M)}: \un{E} \longrightarrow \lim \left(
\begin{tikzcd}
	{\un{c}M} && {\un{T}E} && {\un{c}M} \\
	\\
	& {\un{T}M} && {\un{c}E}
	\arrow["{(0_M,1_M)}"', from=1-1, to=3-2]
	\arrow["{(T(q),q)}"', from=1-3, to=3-2]
	\arrow["{(p_E,1_E)}", from=1-3, to=3-4]
	\arrow["{(z,z)}", from=1-5, to=3-4]
\end{tikzcd} \right) \]
In the course of verifying the conditions of Proposition~\ref{prop:lim} we have showed that this linear map is an isomorphism on both total objects and base objects. Therefore the map is an isomorphism of differential bundles, which completes the proof of Theorem~\ref{thm:main}.
\end{proof}

We can now state our full characterization of differential bundles in an arbitrary tangent category, based on Proposition~\ref{prop:lim} and Theorem~\ref{thm:main}.

\begin{corollary} \label{cor}
There is a one-to-one correspondence between differential bundles in a tangent category $\bcat{X}$ and triples $(q,z,\lambda)$ of morphisms in $\bcat{X}$ such that:
\begin{enumerate} \itemsep=10pt
\item $zq = 1_M$;
\item the wide pullback of $k$ copies of $q$ over $M$, which we denote $q_k: E_k \to M$, exists for all $k \geq 2$, and that wide pullback is preserved by $T^n$ for all $n \geq 1$;
\item for each $k,n \geq 0$, there is an isomorphism
\[ \pb{T^n(q_k),T^n(\lambda_k),T^n(q_k)}: T^n(E_k) \stackrel{\isom}{\longrightarrow} T^n(M) \times_{T^n(T_kM)} T^n(T_kE) \times_{T^n(E)} T^n(M) \]
where the maps in the double pullback are
\[ T^n((0_k)_M), \; T^n(T_k(q)), \; T^n((p_k)_E), \text{ and }T^n(z), \]
and $\lambda_k: E_k \to T_kE$ is the map on wide pullbacks induced by $k$ copies of $\lambda$ (over $z$).
\end{enumerate}
\end{corollary}

\begin{remark} \label{rem:cor2}
If the category $\bcat{X}$ has all pullbacks, and the functor $T$ preserves pullbacks, then (2) and (3) in Corollary~\ref{cor} can be replaced by the single condition that there is an induced isomorphism
\[ \pb{q,\lambda,q}: E \to M \times_{TM} TE \times_E M. \]
\end{remark}

\begin{remark}
Corollary~\ref{cor} should be compared to MacAdam's characterization of differential bundles in~\cite[Proposition 6]{macadam:2021}.

It follows from Theorem~\ref{thm:main} that the isomorphism class of a differential bundle is determined by its projection and zero section. The vertical lift then plays the role of singling out a specific bundle within that isomorphism class. Notice that the addition operation is therefore completely determined by the rest of the structure. (This fact was also observed by MacAdam~\cite[Lemma 5]{macadam:2021}.) To be explicit, we can determine the addition map $\sigma: E \times_M E \to E$ by pulling back the additive stucture on the differential bundle in Proposition~\ref{prop:lim}. That gives
\[ \sigma = \pb{q_2,\lambda_2 (+_E)} i_1^{-1} \]
where $i_1 = \pb{q,\lambda}: E \to M \times_{TM} TE \times_E M$ is the isomorphism of Theorem~\ref{thm:main}.
\end{remark}

\section{Linear Maps} \label{sec:linear}

Now that we have characterized the differential bundles in a tangent category, it makes sense to consider the linear maps between those bundles, i.e.\ those maps of bundles which commute with the vertical lift. Unfortunately, our approach does not appear to give a simple classification of linear maps; let us explain why.

\begin{definition} \label{def:linear}
Let $(g,f): (E,M,q,z) \to (E',M',q',z')$ be a map between bundles which commutes with the projections and sections, i.e. we have commutative diagrams
\[\begin{tikzcd}
	E & {E'} & E & {E'} \\
	M & {M'} & M & {M'}
	\arrow["g", from=1-1, to=1-2]
	\arrow["q"', from=1-1, to=2-1]
	\arrow["{q'}", from=1-2, to=2-2]
	\arrow["g", from=1-3, to=1-4]
	\arrow["f"', from=2-1, to=2-2]
	\arrow["z", from=2-3, to=1-3]
	\arrow["f"', from=2-3, to=2-4]
	\arrow["{z'}"', from=2-4, to=1-4]
\end{tikzcd}\]
Then, assuming the relevant pullbacks exist and are preserved by $T$, the pair $(g,f)$ determines a commutative diagram
\begin{equation} \label{eq:linear} \begin{tikzcd}
	{M \times_{TM} TE \times_E M} &&& {M' \times_{TM'} TE' \times_{E'} M'} \\
	M &&& {M'}
	\arrow["{f \times_{T(f)} T(g) \times_g f}", from=1-1, to=1-4]
	\arrow[from=1-1, to=2-1]
	\arrow[from=1-4, to=2-4]
	\arrow["f"', from=2-1, to=2-4]
\end{tikzcd} \end{equation}
which is a linear map of differential bundles: the top map commutes with the vertical lifts by the naturality of the vertical lift map $\ell$ in the tangent structure (which determines the vertical lifts for these bundles).
\end{definition}

However, the construction of Definition~\ref{def:linear} does not, in general, produce all linear maps between the differential bundles constructed in this way (unless $q$ and $q'$ are already differential bundles to begin with). Here is a counterexample based on the tangent category $\CRing$ of commutative rings described in Example~\ref{ex:cring} below.

\begin{example} \label{ex:fail}
Let $p: R[\epsilon] \to R$ be the ring map given by $\epsilon \mapsto 0$, where $\epsilon^2 = 0$, with zero section $0: R \to R[\epsilon]$ given by $0(r) = r$. (Thus, $p = p_R$ and $0 = 0_R$ are the projection and zero section for the tangent bundle on $R$ in this tangent structure.) Let $q: R[X] \to R$ be the corresponding map for the polynomial ring, i.e. given by $X \mapsto 0$ with section $z: R \to R[X]$ given by $z(r) = r$.

We will see in Example~\ref{ex:cring} that the differential bundles associated to these data take the form of the square-zero extensions
\[\begin{tikzcd}
	{R \oplus R\epsilon} && {R \oplus M\epsilon} \\
	R && R
	\arrow["p"', from=1-1, to=2-1]
	\arrow["{q'}"', from=1-3, to=2-3]
\end{tikzcd}\]
In the first case this is the same as the original $p$ since we started with a differential bundle (the tangent bundle on $R$), but in the second case it is not: $M := \ker(q)$ is the augmentation ideal, i.e. the ideal in $R[X]$ consisting of polynomials with zero constant term.

We also show in Example~\ref{ex:cring} that the linear maps (over $1_R$) from $p$ to $q'$ correspond to the $R$-module homomorphisms $\phi: R \to M$, which in turn correspond to elements $m := \phi(1) \in M$. In particular, there is a linear map $\phi: R \oplus R\epsilon \to R \oplus M\epsilon$ such that $\phi(\epsilon) = X\epsilon$.

On the other hand, $R$-algebra homomorphisms $g: R[\epsilon] \to R[X]$ correspond to elements $g(\epsilon) \in R[X]$ with the property that
\[ g(\epsilon)^2 = g(\epsilon^2) = g(0) = 0. \]
Therefore, there is no such $g$ with the property that $g(\epsilon) = X$, and hence no map of bundles $R[\epsilon] \to R[X]$ which induces the linear map $\phi$ above.
\end{example}

Thus, our perspective on differential bundles does not appear to provide any particular classification of the linear maps between those bundles. In order to calculate the linear maps in the examples in the next section, we will usually work directly from the definition~\cite[2.3]{cockett/cruttwell:2018}: that is, a linear map is a map of bundles which also commutes with the vertical lifts. One benefit our approach does have is that Proposition~\ref{prop:lim} provides an explicit formula for the vertical lift, based on the vertical lift map $\ell$ in the underlying tangent structure.

\begin{remark}
Let $\mathrm{Bun}(\bcat{X})$ be the category in which an object is a morphism $q: E \to M$ with a section $z: M \to E$, and a morphism is a pair $(g,f)$ as in Definition~\ref{def:linear}. Then, subject to the existence (and preservation by $T^n$) of the relevant pullbacks, the construction of Proposition~\ref{prop:lim} determines a functor
\[ D: \mathrm{Bun}(\bcat{X}) \to \DBun(\bcat{X})_{\mathrm{lin}}. \]
Theorem~\ref{thm:main} then provides a natural isomorphism
\[ \un{\lambda}: 1_{\DBun(\bcat{X})_{\mathrm{lin}}} \isom DU \]
where $U: \DBun(\bcat{X})_{\mathrm{lin}} \to \mathrm{Bun}(\bcat{X})$ is the forgetful functor. 

One might wonder if $\un{\lambda}$ is the unit or counit for an adjunction between $D$ and $U$, but in general that is not the case. For example, when $\bcat{X} = \CRing$ with the tangent structure of Example~\ref{ex:cring}, $\mathrm{Bun}(\bcat{X})$ is equivalent to the category of pairs $(R,E)$ consisting of a ring $R$ and an augmented $R$-algebra $E$. As we will see below, $\DBun(\bcat{X})_{\mathrm{lin}}$ is equivalent to the category of pairs $(R,M)$ consisting of a ring $R$ and an $R$-module $M$. With respect to these equivalences, the functor $D$ can be identified with the functor
\[ (R,E) \mapsto (R,\bar{E}) \]
where $\bar{E} = \ker(E \to R)$ is the augmentation ideal. That augmentation ideal functor has a left adjoint given by the symmetric algebra functor
\[ (R,M) \mapsto (R,\Sym_R(M)) \]
and no right adjoint (because it does not preserve colimits). On the other hand, the forgetful functor $U$ corresponds, under these equivalences, to the square-zero extension functor
\[ (R,M) \mapsto (R, R \oplus M \epsilon), \]
which in turn has left adjoint given by forming the module of indecomposables, the quotient of the augmentation ideal by its square.

The categories $\mathrm{Bun}(\bcat{X})$ and $\DBun(\bcat{X})_{\mathrm{lin}}$ are each themselves a tangent category with tangent bundle functor given on the projection and zero section by
\[ T(q,z) = (T(q),T(z)). \]
For $\mathrm{Bun}(\bcat{X})$ this claim follows from~\cite[Ex. 2.2(vii)]{cockett/cruttwell:2018} since $\mathrm{Bun}(\bcat{X})$ is a category of diagrams in $\bcat{X}$. For $\DBun(\bcat{X})_{\mathrm{lin}}$ it is proved in the discussion following~\cite[5.7]{cockett/cruttwell:2018}. It appears that the functors $U$ and $D$ are then each a strong morphism of tangent structure in the sense of~\cite[2.7]{cockett/cruttwell:2014}. For $U$ that claim is immediate from the definition of the tangent structure on $\DBun(\bcat{X})_{\mathrm{lin}}$. For $D$ we have to provide a natural isomorphism of differential bundles (over $TM$):
\[ T(M \times_{TM} TE \times_E M) \isom TM \times_{T^2M} T^2E \times_{TE} TM \]
which comes from the assumption (made in the construction of $D$) that $T$ preserves this double pullback.
\end{remark}

\section{Examples} \label{sec:ex}

We conclude the paper with some examples of how our characterization can be used to understand differential bundles in various tangent categories, as well as how to determine particular kinds of differential bundles, such as differential objects.

\begin{example}[Manifolds] \label{ex:mfld}
We can use Theorem~\ref{thm:main} to recover MacAdam's result~\cite[Proposition 12]{macadam:2021} that differential bundles in $\Mfld$ are smooth vector bundles. Note that we follow the definition of smooth manifold in which different connected components of a manifold are required to have the same dimension. In~\cite{macadam:2021} a more general definition is used without that restriction. In that more general case differential bundles are vector bundles which are permitted to have different fibre dimensions over different connected components.

Let $q: E \to M$ be a differential bundle in the tangent category $\Mfld$ of smooth manifolds and smooth maps. Then, by Theorem~\ref{thm:main}, $q$ is diffeomorphic to
\[ q': M \times_{TM} TE \times_E M \to M. \]
The tangent bundle $p_E: TE \to E$ is a smooth vector bundle, and since smooth vector bundles are stable under pullback~\cite[{\S}3]{milnor/stasheff:1974}, it follows that
\[ TE \times_E M \to M \]
is a smooth vector bundle.

We now claim that the map
\[ f: TE \times_E M \to TM; \quad (v,m) \mapsto T(q)(v) \]
is a linear map between smooth vector bundles over $M$, i.e.\ that the map induced by $f$ on each fibre is linear. On the fibre over $m \in M$, that map is given by the derivative of $q$:
\[ f_m =  dq_{z(m)}: T_{z(m)}E \to T_mM \]
which is linear as required.

Moreover, each linear map $f_m$ is surjective because it has a section given by $dz_m$, so $f$ is a map of smooth vector bundles of constant rank (equal to $\dim(E) - \dim(M)$). Therefore, by~\cite[10.34]{lee:2013}, the kernel of $f$ is a smooth vector bundle over $M$. That kernel consists of those tangent vectors to $z(m)$ which project to $0$ in $T_mM$, which is precisely the pullback
\[ q': M \times_{TM} TE \times_E M \to M. \]
Therefore $q'$ is a smooth vector bundle over $M$, and hence so too is the original differential bundle $q: E \to M$. This verifies that differential bundles in $\Mfld$ are smooth vector bundles. Along with the (easier) proof that smooth vector bundles are differential bundles in $\Mfld$~\cite[Proposition 11]{macadam:2021}, we recover MacAdam's result that these two notions coincide.
\end{example}

\begin{example}[Trivial tangent categories] \label{ex:triv}
Let $\bcat{X}$ be a category with the trivial tangent structure in which $T$ is the identity functor, and all tangent bundle structure maps are identity morphisms. Let $q: E \to M$ be a morphism with section $z: M \to E$. The corresponding differential bundle is given by the double pullback
\[ M \times_M E \times_E M \to M \]
which is an isomorphism. Thus every differential bundle in $\bcat{X}$ is trivial, i.e.\ isomorphic to the bundle $\un{c}M$ for some $M$.
\end{example}

\begin{example}[Trivial differential bundles] \label{ex:triv-bdle}
Let $M$ be any object in a tangent category, and let $q, z$ both be the identity morphism on $M$. Proposition~\ref{prop:lim} determines a differential bundle of the form
\[ M \times_{TM} TM \times_M M \to M \]
which is isomorphic to the trivial differential bundle over $M$. This fact also follows directly from Theorem~\ref{thm:main} applied to the trivial bundle $\un{c}M$.
\end{example}

\begin{example}[Differential objects] \label{ex:diffobj}
Let $\bcat{X}$ be a \emph{cartesian} tangent category, i.e.\ such that $\bcat{X}$ has finite products which are preserved by $T$. Then Cockett and Cruttwell~\cite[4.3]{cockett/cruttwell:2014} define the notion of \emph{differential object} in $\bcat{X}$, and they show in~\cite[3.4]{cockett/cruttwell:2018} that differential objects correspond precisely to differential bundles over the terminal object $1 \in \bcat{X}$. We can therefore use Theorem~\ref{thm:main} to give the following alternative description of differential objects.

Let $q:E \to 1$ be the unique map from $E$ to the terminal object, and let $z: 1 \to E$ be any morphism, i.e.\ a generalized point of $E$ in $\bcat{X}$. Then, according to Proposition~\ref{prop:lim}, there is a differential bundle in $\bcat{X}$ of the form
\[ 1 \times_{T(1)} T(E) \times_{E} 1 \to 1 \]
as long as the double pullback exists and is preserved by $T^n$. (In this case, the wide pullbacks of the projection are simply finite products in $\bcat{X}$, which we have already assumed to exist and be preserved by $T^n$.)

Since $T(1) \isom 1$ in a cartesian tangent category, the double pullback above is precisely the \emph{tangent space} $T_zE$ to the object $E$ at the point $e$. Therefore, our theory recovers the fact~\cite[4.15]{cockett/cruttwell:2014} that every tangent space in a tangent category has a canonical differential structure. In fact, if one follows through Cockett and Cruttwell's proof of that fact, we see that it essentially agrees with the proof we have given here for differential bundles in general.

Theorem~\ref{thm:main} thus gives us the following characterization of differential objects in a cartesian tangent category: a differential object consists of an object $E$ and maps $z,\lambda$ for which the following is a pullback that is preserved by $T^n$ for all $n \geq 1$:
\[\begin{tikzcd}
	E & TE \\
	1 & E
	\arrow["\lambda", from=1-1, to=1-2]
	\arrow[from=1-1, to=2-1]
	\arrow["{p_E}", from=1-2, to=2-2]
	\arrow["z"', from=2-1, to=2-2]
\end{tikzcd}\]
Of course, that characterization amounts simply to saying $E$ is isomorphic to one of its tangent spaces. 
\end{example}

\begin{example}[Cartesian differential categories] \label{ex:CDC}
One source of tangent categories, described in~\cite[4.7]{cockett/cruttwell:2014}, is the theory of cartesian differential categories (CDCs) of Blute, Cockett, and Seely~\cite{blute/cockett/seely:2009}. These examples can essentially be viewed as cartesian tangent categories $\bcat{X}$ in which every object is equipped with a chosen differential structure. In particular, the tangent bundle functor $T: \bcat{X} \to \bcat{X}$ is given by
\[ T(A) = A \times A. \]
Note that $T$ preserves all pullbacks in $\bcat{X}$. We can now identify differential bundles in $\bcat{X}$ as follows. (We are grateful to JS Lemay for helping to clarify some of the following arguments.)

Let $q: E \to M$ be a morphism with section $z: M \to E$. We can rewrite the double pullback which determines the corresponding differential bundle in the form
\[ (M \times 1) \times_{(M \times M)} (E \times E) \times_{(E \times 1)} (M \times 1) \]
where we use the fact that $0_M: M \to M \times M$ is given by $\pb{1_M,*}$, and $p_E: E \times E \to E$ is the second projection. (Recall that in a CDC each hom-set has the structure of a commutative monoid. Here $*$ denotes the zero map $M \to M$, which factors via the terminal object $1$.)

Since a CDC has finite products, we can rewrite the double pullback above as a single pullback
\begin{equation} \label{eq:CDC} (M \times 1) \times_{(M \times M)} (M \times E). \end{equation}
In general, we cannot say more; this pullback need not exist in $\bcat{X}$. However, if it does then its projection to $M \isom (M \times 1)$ is a differential bundle, and every differential bundle is of this form. 

Suppose, however, that the map $q: E \to M$ has a kernel, i.e. there is a pullback square in $\bcat{X}$ of the form
\[\begin{tikzcd}
	{E_0} & E \\
	1 & M
	\arrow["i", from=1-1, to=1-2]
	\arrow[from=1-1, to=2-1]
	\arrow["q", from=1-2, to=2-2]
	\arrow["0"', from=2-1, to=2-2]
\end{tikzcd}\]
where $0: 1 \to M$ is the zero map.

In that case, the pullback (\ref{eq:CDC}) is given by $M \times E_0$, and the differential bundle corresponding to the original pair $(q,z)$ is the projection map
\[ M \times E_0 \to M. \]
Conversely, any projection map $q: M \times F \to M$ is a differential bundle with zero section
\[ z := \pb{1_M,0}: M \to M \times F \]
and vertical lift
\[ \lambda = \pb{1_M \times 0, 0 \times 1_F} : (M \times F) \to (M \times F) \times (M \times F). \]
Thus the differential bundles in the canonical tangent structure on a CDC include the product projection maps, which can be thought of as the trivial differential bundles with arbitrary fibre $F$, though there can be additional differential bundles of the form (\ref{eq:CDC}).
\end{example}

\begin{example}[Synthetic differential geometry (SDG)] \label{ex:sdg}
By~\cite[5.4]{cockett/cruttwell:2014} the microlinear objects in a model of SDG form a cartesian tangent category $\bcat{X}$ in which the tangent bundle functor is given by exponential objects
\[ T(X) = X^D \]
associated to the object $D$ of infinitesimals. It is shown in~\cite[Ex. 2.4(iv)]{cockett/cruttwell:2018} that differential bundles in this tangent category are the `Euclidean vector bundles' over $M$; see~\cite[3.1.2, Def. 2]{lavendhomme:1996}. We can use Corollary~\ref{cor} to recover this result.

The tangent category $\bcat{X}$ admits all pullbacks, and the functor $T$ is a right adjoint so preserves those pullbacks. By Remark~\ref{rem:cor2} a differential bundle in $\bcat{X}$ therefore consists of a triple $(q,z,\lambda)$ such that $q: E \to M$ has section $z: M \to E$, and $\lambda: E \to E^D$ determines an isomorphism
\[ \pb{q,\lambda,q}: E \isom M \times_{M^D} E^D \times_E M. \]
Transpose to $\lambda$ is a map $D \times E \to E$, which we write as an action of $D$ on $E$ by $(d,e) \mapsto d \cdot e$.

We can then express the necessary conditions on $q,z,\lambda$ as follows using the internal language for the underlying topos:
\begin{itemize} \itemsep=10pt
\item $q(z(m)) = m$ for all $m \in M$;
\item $0 \cdot e = z(q(e))$ for all $e \in E$;
\item $q(d \cdot e) = q(e)$ for all $d \in D$ and all $e \in E$;
\item for any morphism $f: D \to E$ such that $q(f(d)) = q(f(0))$ for all $d \in D$, and $f(0) = z(q(f(0))$, there is a unique $e \in E$ such that 
\[ f(d) = d \cdot e \quad \text{for all $d \in D$}. \]
\end{itemize}
These conditions are a subset of the set of requirements for $q$ to be a Euclidean vector bundle in the sense of~\cite[3.1.2, Def. 2]{lavendhomme:1996}, i.e. that $q$ has the structure of a Euclidean $R$-module (see~\cite[1.1.4, Def. 1]{lavendhomme:1996}) on each fibre $E_m := q^{-1}(m)$, where $R$ is the line object in $\bcat{X}$. Therefore every Euclidean vector bundle is a differential bundle.

Conversely, let $q: E \to M$ be the projection map of a differential bundle in $\bcat{X}$. Then, by Theorem~\ref{thm:main}, $q$ is isomorphic to the projection
\[ M \times_{TM} TE \times_E M \to M. \]
To see that $q$ is a Euclidean vector bundle, we have to show that each fibre is a Euclidean $R$-module. We can identify the fibre of this projection over $m$ as the kernel of the derivative map
\[ dq_{z(m)}: T_{z(m)}E \to T_mM \]
By~\cite[3.1.2, Prop. 2,3]{lavendhomme:1996}, that is a linear map between Euclidean $R$-modules, so it is sufficient to show that the kernel of such a map is also a Euclidean $R$-module.

So let $\alpha: N \to N'$ be an $R$-linear map between Euclidean $R$-modules with kernel $K$, and let $f: D \to K$ be any morphism. Since $N$ is Euclidean, there is a unique $b \in N$ such that
\[ f(d) = f(0) + d \cdot b \]
for all $d \in D$. We have
\[ d \cdot \alpha(b) = \alpha(d \cdot b) = \alpha(f(0) + d \cdot b) = \alpha(f(d)) = 0. \]
Since $N'$ is Euclidean, it follows that $\alpha(b) = 0$, and so $b$ is in $K$. Thus $K$ is Euclidean as required.

Therefore, differential bundles in $\bcat{X}$ are precisely the Euclidean vector bundles.
\end{example}

\begin{example}[Commutative rings] \label{ex:cring}
Let $\CRing$ be the tangent category of commutative rings with identity, with tangent bundle functor given by
\[ T(R) = R[\epsilon] \]
where $\epsilon^2 = 0$. See~\cite[3.1]{cruttwell/lemay:2023} for a complete description of this tangent structure. We recover the description of differential bundles in $\CRing$ due to Cruttwell and Lemay~\cite[3.13]{cruttwell/lemay:2023} as follows.

Let $q: E \to R$ be a ring homomorphism with section $z: R \to E$. In other words, $E$ is an augmented $R$-algebra. The category $\CRing$ has all limits, and $T$ is a right adjoint, so preserves those limits. Therefore the conditions of Proposition~\ref{prop:lim} are automatically satisfied, and so there is a differential bundle over $R$ of the form
\[ V = R \times_{R[\epsilon]} E[\epsilon] \times_E R. \]
That pullback is the ring with elements
\[ (r,z(r)+b\epsilon,r) \]
where $b \in \ker(q)$, with ring structure determined by those in $R$ and $E$ with $\epsilon^2 = 0$. In other words, we have
\[ V \isom R \oplus M\epsilon, \]
the square-zero extension of the commutative ring $R$ by the $R$-module $M := \ker(q)$. Thus every differential bundle is, up to isomorphism, the projection map
\[ R \oplus M\epsilon \to R \]
for some such square-zero extension.

Conversely, given an arbitrary $R$-module $M$, apply the previous construction to the projection homomorphism $q: R \oplus M\epsilon \to R$. We can identify $\ker(q) \isom M$ as $R$-modules, and so $q$ is a differential bundle. Thus the differential bundles are, up to isomorphism, precisely the square-zero extensions of $R$ by an $R$-module $M$.

Note that the vertical lift of this square-zero extension is the map
\[ \lambda: (R \oplus M[\epsilon]) \to (R \oplus M[\epsilon])[\epsilon_1],\]
where $\epsilon_1^2 = 0$ also, given by
\[ \lambda(r+m\epsilon) = (r+0\epsilon)+(0+m\epsilon)\epsilon_1. \]
A ring homomorphism $f: R \oplus M\epsilon \to R \oplus M'\epsilon$ over $R$ commutes with the vertical lifts if and only if it is induced by an $R$-module homomorphism $M \to M'$, and so we have identified an equivalence of categories
\[ \DBun_R(\CRing)_{\mathrm{lin}} \simeq \Mod_R \]
between the category of differential bundles over $R$ in $\CRing$ and the category of $R$-modules.
\end{example}

\begin{example}[Commutative rings (opposite)] \label{ex:cring-op}
There is a tangent structure on the opposite of the category of commutative rings, in which the tangent bundle functor $S: \CRing^{op} \to \CRing^{op}$ is the (opposite of) the left adjoint of the tangent bundle functor $T: \CRing \to \CRing$ of Example~\ref{ex:cring}. Explicitly, the functor $S$ is given by
\[ S(R) = \mathrm{Sym}_R(\Omega_R), \]
the free commutative $R$-algebra on the $R$-module $\Omega_R$ of K\"{a}hler differentials of $R$ over $\Z$. A presentation of $S(R)$ as a commutative $R$-algebra is given as follows: we have algebra generators $d[r]$ for each $r \in R$, together with relations
\[ d[0] = 0, \quad d[1] = 0, \quad d[r+s] = d[r] + d[s], \quad d[rs] = r d[s]+s d[r]. \]
A full description of this tangent structure on $\CRing^{op}$ is given in~\cite[4.2]{cruttwell/lemay:2023}.

We now apply Theorem~\ref{thm:main} to recover Cruttwell and Lemay's description~\cite[4.17]{cruttwell/lemay:2023} of the differential bundles in the tangent category $\CRing^{op}$.

Let $q^{op}: E \to R$ be a morphism in $\CRing^{op}$ with section $z^{op}: R \to E$. In other words, we have ring homomorphisms $q: R \to E$ and $z: E \to R$ such that $qz = 1_R$; again, these maps make $E$ into an augmented commutative $R$-algebra, though note that the roles of $q$ and $z$ are reversed. Since $\CRing^{op}$ has limits and $S: \CRing^{op} \to \CRing^{op}$ is a right adjoint (its opposite $S^{op}$ is left adjoint to $T$), we can apply Proposition~\ref{prop:lim} to obtain a differential bundle in $\CRing^{op}$ whose projection map, viewed in $\CRing$ takes the form
\[ q': R \to R \otimes_{S(R)} S(E) \otimes_E R \]
where $\otimes$ denotes the coproduct (and pushout) in $\CRing$, which is given by the (relative) tensor product of commutative rings. In other words, the target of $q'$ is the commutative ring
\[ R \otimes_{\Sym_R(\Omega_R)} \Sym_E(\Omega_E) \otimes_{E} R. \]
This ring is generated as an $R$-algebra by the symbols $d[e]$ for $e \in E$, with relations
\[ d[e+e'] = d[e]+d[e'], \quad d[ee'] = z(e) d[e'] + z(e') d[e], \quad d[q(r)] = 0.\]
Notice that the other relations $d[0] = 0$ and $d[1] = 0$ in $\Omega_E$ are subsumed by the last one. Since all of these relations are homogenous in the number of symbols $d[-]$, the corresponding ring is a free commutative $R$-algebra on the $R$-module with those same relations and generators. That $R$-module is precisely the extension of scalars (along $z: E \to R$) of the $E$-module of \emph{relative} K\"{a}hler differentials $\Omega_{E/R}$. We therefore have
\[ R \otimes_{\Sym_R(\Omega_R)} \Sym_E(\Omega_E) \otimes_{E} R \isom \Sym_R(R \otimes_E \Omega_{E/R}). \]
The differential bundle in $\CRing^{op}$ constructed from $q^{op}$ and $z^{op}$ therefore takes the form
\[ q': R \to \Sym_R(R \otimes_E \Omega_{E/R}). \]
In particular, this is the inclusion of scalars in a \emph{free} commutative $R$-algebra on an $R$-module $M = R \otimes_E \Omega_{E/R}$.

Now take $E = R \oplus M\epsilon$ for an arbitrary $R$-module $M$, with $\epsilon^2 = 0$. There is an isomorphism of $R$-modules
\[ R \otimes_{(R \oplus M\epsilon)} \Omega_{(R \oplus M\epsilon)/R} \isom M; \quad r \otimes d[m\epsilon] \mapsto r \cdot m, \]
and so we deduce that there is a differential bundle in $\CRing^{op}$ with projection map the ring homomorphism
\[ R \to \Sym_R(M). \]
Hence a differential bundle in $\CRing^{op}$ is, up to isomorphism, the inclusion of scalars
\[ q': R \to \Sym_R(M) \]
for an $R$-module $M$.

Now let's determine the linear maps between these differential bundles. Those are the $R$-algebra homomorphisms
\[ \phi: \Sym_R(M) \to \Sym_R(M') \]
which commute with the vertical lifts on each side. The vertical lift for the bundle
\[ q': R \to R \otimes_{S(R)} S(E) \otimes_E R \]
is induced by the vertical lift for the tangent structure: that is the morphism in $\bcat{X}^{op}$ corresponding to
\[ \ell_E: S^2(R) \to S(R) \]
given by (according to~\cite[4.2(vii)]{cruttwell/lemay:2023}):
\[ \ell_E(e) = e, \quad \ell_E(d[e]) = 0, \quad \ell_E(d'[e]) = 0, \quad \ell_E(d'[d[e]]) = e. \]
Translating into the vertical lift for the differential bundle
\[ q': R \to \Sym_R(M) \]
we get
\[ \lambda: S(\Sym_R(M)) \to \Sym_R(M) \]
given on $R$-algebra generators by
\[ \lambda(m) = 0, \quad \lambda(d[r]) = 0, \quad \lambda(d[m]) = m \]
for $m \in M$ and $r \in R$; see also~\cite[4.10]{cruttwell/lemay:2023}. In particular, we have, for $r \in R$, $m_1,\dots,m_k \in M$:
\[ \lambda(d[r \cdot m_1 \cdots m_k]) = \sum_{i = 1}^{k} r \lambda(m_1) \cdots \lambda(d[m_j]) \cdots \lambda(m_k) \]
which is equal to $0$ unless $k = 1$, in which case we have
\[ \lambda(r d[m_1]) = r \cdot m_1 \in M. \]
An arbitrary element $v \in \Sym_R(M)$ is a finite sum of elements of the form $r \cdot m_1 \dots m_k$, for varying $k$. It follows that for any $v \in \Sym_R(M)$, we have
\begin{equation} \label{eq:d} \lambda(d[v]) \in M. \end{equation}

Now consider an $R$-algebra homomorphism $\phi: \Sym_R(M) \to \Sym_R(M')$ for $R$-modules $M,M'$. We claim that $\phi$ commutes with the vertical lifts if and only if $\phi(M) \subseteq M'$.

First suppose $\phi(M) \subseteq M'$. We prove that $S(\phi)\lambda' = \lambda\phi$ by checking this equation on each of the algebra generators of $S(\Sym_R(M))$:
\[ \begin{split} & \lambda'(S(\phi))(m) = \lambda'(\phi(m)) = 0, \quad \phi(\lambda(m)) = \phi(0) = 0, \\
	& \lambda'(S(\phi)(d[r])) = \lambda'(d[\phi(r)]) = \lambda'(d[r]) = 0, \quad \phi(\lambda(d[r])) = \phi(0) = 0, \\
	& \lambda'(S(\phi)(d[m])) = \lambda'(d[\phi(m)]) = \phi(m), \quad \phi(\lambda(d[m])) = \phi(m). \end{split} \]
Since both sides are equal on each algebra generator, we deduce that $\phi$ commutes with the vertical lifts, so $\phi$ is a linear map of differential bundles.

Conversely, suppose $S(\phi)\lambda' = \lambda\phi$ and take an arbitrary $m \in M$. Then we have
\[ \phi(m) = \phi(\lambda(d[m])) = \lambda'(S(\phi)(d[m])) = \lambda'(d[\phi(m)]). \]
By (\ref{eq:d}), it follows that $\phi(m) \in M'$. Note that since $\phi(rm) = \phi(r)\phi(m) = r\phi(m)$, we also deduce that the restriction of $\phi$ to $M$ is an $R$-module homomorphism.

We have therefore identified the linear maps of differential bundles over $R$ of the form
\[ \phi: \Sym_R(M) \to \Sym_R(M') \]
with those $R$-algebra homomorphisms induced by an $R$-module homomorphism $\phi: M \to M'$. It follows that there is an equivalence of categories
\[ \DBun_R(\CRing^{op}) \simeq \Mod_R^{op} \]
recovering the Cruttwell-Lemay description in~\cite[4.17]{cruttwell/lemay:2023}.
\end{example}

\begin{example}[Schemes] \label{ex:schemes}
Identifying $\CRing^{op}$ with the category of affine schemes, there is an extension of the tangent structure considered in Example~\ref{ex:cring-op} to the category $\Sch$ of all schemes (and scheme morphisms). That extension is given by applying the functor $S$ of Example~\ref{ex:cring-op} affine locally. Cruttwell and Lemay~\cite[4.28]{cruttwell/lemay:2023} use the result for affine schemes to show that differential bundles in $\Sch$ over a scheme $X$ correspond (contravariantly) to quasi-coherent sheaves of $\mathcal{O}_X$-modules. In this case, there does not seem to be any benefit to approaching the classification of differential bundles directly using Theorem~\ref{thm:main}.
\end{example}

\begin{example}[Algebras over an operad] \label{ex:operad}
Let $R$ be a commutative ring with identity, and let $P$ be an operad in the symmetric monoidal category of $R$-modules. Ikonicoff, Lanfranchi, and Lemay~\cite{ikonicoff/lanfranchi/lemay:2023} have constructed dual tangent structures on the category $\Alg_P$ of $P$-algebras (in $R$-modules) and its opposite $\Alg_P^{op}$. Those structures specialize to the tangent structures on Examples~\ref{ex:cring} and~\ref{ex:cring-op} in the case that $R = \Z$ and $P$ is the commutative operad. In further work, Lanfranchi~\cite[4.8]{lanfranchi:2023} has determined the differential bundles in the tangent category $\Alg_P^{op}$, showing that differential bundles over a $P$-algebra $A$ correspond, up to isomorphism, to $A$-modules in the operadic sense. As noted, this generalizes the classification in Example~\ref{ex:cring-op}. 

The methods of Example~\ref{ex:cring-op} can presumably be modified to treat the case of the tangent structure on $\Alg_P^{op}$, replacing the free commutative (symmetric) algebra construction with a free $A$-algebra construction for a $P$-algebra $A$~\cite[4.4]{ikonicoff/lanfranchi/lemay:2023}, and replacing the module of K\"{a}hler differentials $\Omega_A$ with the corresponding module for a $P$-algebra $A$; see~\cite[12.3.8]{loday/vallette:2012}. We expect to be able to recover Lanfranchi's result by this process.

We are not aware of an explicit description of the differential bundles in the tangent structure on $\Alg_P$, though this case seems much more straightforward than that treated in~\cite{lanfranchi:2023}, and the correct answer was conjectured by Ikonicoff, Lanfranchi, and Lemay in~\cite[5(ii)]{ikonicoff/lanfranchi/lemay:2023}. By modifying the calculations of Example~\ref{ex:cring} as needed, we show that there is an equivalence of categories
\[ \DBun_A(\Alg_P) \simeq \Mod^P_A \]
between differential bundles over $A$, and the category of $A$-modules, in the operadic sense, for a $P$-algebra $A$. 

Recall from~\cite{ikonicoff/lanfranchi/lemay:2023} that the tangent bundle functor $T: \Alg_P \to \Alg_P$ is given by
\[ T(A) := A \ltimes A. \]
The underlying $R$-module of $A \ltimes A$ is $A \times A$, and the $P$-algebra structure map
\[ \phi^n_{T(A)}: P(n) \otimes T(A)^{\otimes n} \to T(A) \]
is given by
\begin{equation} \label{eq:mu} (\mu, (a_1,b_1), \dots, (a_n,b_n)) \mapsto \left( \phi^n_A(\mu, a_1, \dots, a_n), \sum_{i = 1}^{n} \phi^n_A(\mu, a_1, \dots, a_{i-1},  b_i, a_{i+1}, \dots, a_n \right) \end{equation}
for $\mu \in P(n)$, $a_1,b_1,\dots,a_n,b_n \in A$, where $\phi^n_A$ denotes the corresponding structure map for $A$.

Now take a morphism $q: E \to A$ in $\Alg_P$ with a section $z: A \to E$. Then the corresponding differential bundle has total space
\[ A \times_{T(A)} T(E) \times_{E} A. \]
This is the $P$-algebra with elements
\[ A \ltimes \ker(q) := \{(z(a),e) \in E \times E \; | \; q(e) = 0 \} \]
and with $P$-algebra structure inherited from that on $E \ltimes E$ by the formula in (\ref{eq:mu}). The kernel $M := \ker(q)$ is an $A$-module in the sense that there are compatible maps of the form
\[ \psi^n_{M,i}: P(n) \otimes A^{i-1} \otimes M \otimes A^{n-i} \to M \]
for $i = 1,\dots,n$, given in this case by restricting $\phi^n_E$ to the subspaces $A \isom z(A)$ and $M = \ker(q)$. Those maps $\psi^n_{M,i}$ make $M$ into an $A$-module.

Conversely, given an $A$-module $M$ with structure maps $\psi^n_{M,i}$, let $E := A \ltimes M$, the $P$-algebra with underlying $R$-module $A \times M$ and $P$-algebra structure maps $\phi^n_E: P(n) \times E^n \to E$ given by
\begin{equation} \label{eq:psi} (\mu, (a_1,m_1), \dots, (a_n,m_n)) \mapsto \left( \phi^n_A(\mu, a_1, \dots, a_n), \sum_{i = 1}^{n} \psi^n_{M,i}(\mu, a_1, \dots,  m_i, \dots, a_n \right). \end{equation}
Let $q:E \to A$ be the projection $q(a,e) = a$, and $z:A \to E$ the map given by $z(a) = (a,0)$. Then $\ker(q)$ is isomorphic to $M$ as a $P$-module, and so there is a differential bundle
\[ q: A \ltimes M \to A. \]
This confirms that every differential bundle is of that form for some $A$-module $M$.

The vertical lift for this differential bundle $q$ is given by
\[ \lambda: A \ltimes M \to (A \ltimes M) \ltimes (A \ltimes M); \quad \lambda(a,m) = ((a,0),(0,m)). \]
It follows, by the same argument as in Example~\ref{ex:cring}, that a map of $P$-algebras over $A$ of the form
\[ \phi: A \ltimes M \to A \ltimes M' \]
commutes with the vertical lifts if and only if $\phi$ is restricts to a map of $A$-modules $\phi: M \to M'$.  Thus, we obtain the desired equivalence of categories
\[ \DBun_A(\Alg_P) \simeq \Mod^P_A. \]
\end{example}

\providecommand{\bysame}{\leavevmode\hbox to3em{\hrulefill}\thinspace}
\providecommand{\MR}{\relax\ifhmode\unskip\space\fi MR }
\providecommand{\MRhref}[2]{%
  \href{http://www.ams.org/mathscinet-getitem?mr=#1}{#2}
}
\providecommand{\href}[2]{#2}


\begin{thebibliography}{BBC23}

\bibitem[BBC23]{bauer/burke/ching:2023}
Kristine Bauer, Matthew Burke, and Michael Ching, \emph{Tangent
  $\infty$-categories and {G}oodwillie calculus},
  \href{www.arxiv.org/abs/2101.07819}{arxiv:2101.07819}.

\bibitem[BCS09]{blute/cockett/seely:2009}
R.~F. Blute, J.~R.~B. Cockett, and R.~A.~G. Seely, \emph{Cartesian differential
  categories}, Theory Appl. Categ. \textbf{22} (2009), 622--672.

\bibitem[CC14]{cockett/cruttwell:2014}
J.~R.~B. Cockett and G.~S.~H. Cruttwell, \emph{Differential structure, tangent
  structure, and {SDG}}, Appl. Categ. Structures \textbf{22} (2014), no.~2,
  331--417.

\bibitem[CC18]{cockett/cruttwell:2018}
\bysame, \emph{Differential bundles and fibrations for tangent categories},
  Cah. Topol. G\'eom. Diff\'er. Cat\'eg. \textbf{59} (2018), no.~1, 10--92.

\bibitem[CL23]{cruttwell/lemay:2023}
G.~S.~H. Cruttwell and Jean-Simon~Pacaud Lemay, \emph{Differential bundles in
  commutative algebra and algebraic geometry}, Theory Appl. Categ. \textbf{39}
  (2023), 1077--1120.

\bibitem[ILL23]{ikonicoff/lanfranchi/lemay:2023}
Sacha Ikonicoff, Marcello Lanfranchi, and Jean-Simon~Pacaud Lemay, \emph{The
  {R}osický tangent categories of algebras over an operad},
  \href{www.arxiv.org/abs/2303.05434}{arxiv:2303.05434}.

\bibitem[Lan23]{lanfranchi:2023}
Marcello Lanfranchi, \emph{The differential bundles of the geometric tangent
  category of an operad},
  \href{www.arxiv.org/abs/2310.18174}{arxiv:2310.18174}.

\bibitem[Lav96]{lavendhomme:1996}
Ren{\'e} Lavendhomme, \emph{Basic concepts of synthetic differential geometry},
  Kluwer Texts Math. Sci., vol.~13, Dordrecht: Kluwer Academic Publishers,
  1996.

\bibitem[Lee13]{lee:2013}
John~M. Lee, \emph{Introduction to smooth manifolds}, 2nd revised ed ed., Grad.
  Texts Math., vol. 218, New York, NY: Springer, 2013.

\bibitem[LV12]{loday/vallette:2012}
Jean-Louis Loday and Bruno Vallette, \emph{Algebraic operads}, Grundlehren
  Math. Wiss., vol. 346, Berlin: Springer, 2012.

\bibitem[LW18]{lucyshyn-wright:2018}
Rory B.~B. Lucyshyn-Wright, \emph{On the geometric notion of connection and its
  expression in tangent categories}, Theory Appl. Categ. \textbf{33} (2018),
  832--866.

\bibitem[Mac21]{macadam:2021}
Benjamin MacAdam, \emph{Vector bundles and differential bundles in the category
  of smooth manifolds}, Appl. Categ. Structures \textbf{29} (2021), no.~2,
  285--310.

\bibitem[MS74]{milnor/stasheff:1974}
John~W. Milnor and James~D. Stasheff, \emph{Characteristic classes}, Ann. Math.
  Stud., vol.~76, Princeton University Press, Princeton, NJ, 1974.

\bibitem[Ros84]{rosicky:1984}
J.~Rosick\'{y}, \emph{Abstract tangent functors}, Diagrammes \textbf{12}
  (1984), JR1--JR11.

\end{thebibliography}

\end{document}